\documentclass[leqno,english,12pt]{amsart}
\usepackage[active]{srcltx}
\usepackage{color}

\usepackage{tikz} 
\usepackage{pgfplots}
\pgfplotsset{compat=1.18} 

\usepackage{epsfig}
\usepackage{amssymb}
\usepackage{latexsym}
\usepackage[latin1]{inputenc}
\usepackage{amsmath,amsthm,amssymb,babel}
\usepackage[colorlinks=true, linkcolor=blue, citecolor=red, urlcolor=blue]{hyperref}
\newtheorem{theorem}{Theorem}

\newtheorem{lemma}{Lemma}
\newtheorem{remark}{Remark}

\def\neweq#1{\begin{equation}\label{#1}}
\def\endeq{\end{equation}}

\def\phi{\varphi}

\begin{document}

\title[Minimum principle for $2$-Hessian problems]{\bf  Minimum principles and a priori estimates for $2$-Hessian problems}


\author[ C. Enache]{Cristian Enache}

\address{Cristian Enache
\hfill\break\indent Department of Mathematics and Statistics
\hfill\break\indent American University of Sharjah 
\hfill\break\indent P.O. Box 26666, Sharjah, United Arab Emirates.
\hfill\break\indent {\tt cenache@aus.edu}}



\date{\today}

\begin{abstract}
{\footnotesize
In this paper we investigate a class of $2$-Hessian equations and establish a minimum principle for a $P$-function in the sense of L.E. Payne (see R. Sperb \cite{Sp81}). The analysis is based on a sharp matrix inequality providing an estimate for a suitable combination of second-order partial derivatives of the solution. Exploiting this estimate, we derive a differential inequality for the associated $P$-function and obtain a minimum principle in higher dimensions under a convexity assumption. As an application of our results, together with convexity results established in X.-N. Ma and L. Xu \cite{MX08}, P. Liu, X.-N. Ma and L. Xu \cite{LMX10}, P. Salani \cite{Sa12}, and Y. Ye \cite{Ye13}, we derive a priori bounds for solutions of several classical $2$-Hessian boundary value problems.
}
\end{abstract}

\maketitle

\section{Introduction}

Let $\Omega \subset \mathbb{R}^{N}$, $N \geq 2$, be a bounded convex domain.
This paper is concerned with a class of Dirichlet boundary value problems
for the $2$-Hessian operator, namely
\begin{equation}
\left\{
\begin{array}{ll}
S_{2}(D^{2}u)=f(u) & \text{in } \Omega, \\[1mm]
u=0 & \text{on } \partial \Omega,
\end{array}
\right.
\tag{1.1}
\label{eq:1.1}
\end{equation}
where $f$ is a positive $C^{1}$ function.
By the strong maximum principle, any classical solution of
\eqref{eq:1.1} satisfying the ellipticity condition is negative in $\Omega$,
a fact that will be implicitly used throughout the paper.

Here $D^{2}u$ denotes the Hessian matrix of $u(\mathbf{x})$, and
$S_{2}(D^{2}u)$ is the $2$-Hessian operator, defined as the second elementary
symmetric function of the eigenvalues $\lambda_{1},\dots,\lambda_{N}$ of
$D^{2}u$, namely
\begin{equation}
S_{2}(D^{2}u)
:= \sum_{1 \le i < j \le N} \lambda_i \lambda_j .
\tag{1.2}
\label{eq:1.2}
\end{equation}
The operator $S_{2}(D^{2}u)$ is homogeneous of degree $2$ and elliptic only
when restricted to the class of $2$-convex functions
\begin{equation}
\Gamma_{2}(\Omega)
:= \left\{ u \in C^{2}(\Omega) :
S_{i}(D^{2}u) \ge 0 \text{ in } \Omega,\ i=1,2 \right\},
\tag{1.3}
\label{eq:1.3}
\end{equation}
where $S_{1}(D^{2}u)=\mathrm{tr}(D^{2}u)=\Delta u$.
Accordingly, problem \eqref{eq:1.1} is said to be elliptic for a given
solution $u(\mathbf{x})$ if the matrix
\begin{equation}
A^{*} := \left( S_{2}^{ij}(D^{2}u) \right),
\qquad
S_{2}^{ij}(D^{2}u)
:= \frac{\partial S_{2}(D^{2}u)}{\partial u_{,ij}},
\tag{1.4}
\label{eq:1.4}
\end{equation}
is positive definite.
Throughout the paper we work exclusively with classical solutions of
\eqref{eq:1.1} satisfying this ellipticity condition, which we refer to as
\emph{admissible solutions}. The existence and qualitative properties of admissible solutions
to problems of type \eqref{eq:1.1} are well understood;
see, for instance, N.S.~Trudinger~\cite{Tr95}.

The study of Hessian equations has a long history, starting with the
foundational works of L.A.~Caffarelli, L.~Nirenberg and J.~Spruck
on fully nonlinear elliptic equations, together with the seminal
contributions of N.S.~Trudinger, or K.S. Chou and X.J. Wang (see, for instance, \cite{CNS84,Tr95, CW01}). More recently, considerable attention has been devoted to convexity and
power-convexity properties of solutions to $2$-Hessian problems in bounded
convex domains.
We refer, among others, to
\cite{MX08,LMX10,Sa12,Ye13},
where it is shown that appropriate monotone transformations
(such as power or logarithmic transforms)
of solutions enjoy strong convexity properties.

Motivated by the classical $P$-function method introduced in the late 1970s by
L.E.~Payne and G.A.~Philippin
(see, for instance, \cite{PP79,Sp81}),
we introduce the auxiliary function

\begin{equation}
\Phi(\mathbf{x},\alpha)
:= |\nabla u|^{2}
+ 2\alpha \int_{u}^{0} f^{1/2}(s)\,ds,
\tag{1.5}
\label{eq:1.5}
\end{equation}
where $u(\mathbf{x})$ is an admissible solution of \eqref{eq:1.1} and
$\alpha \in \mathbb{R}$. Maximum and minimum principles for functions of this type have proved to be a powerful tool in the study of fully nonlinear problems (see, for instance, \cite{PP79,Sp81,Ma99, Ma00, PS04, EMP21} and some references therein). In particular, the following maximum principle is already known:

\begin{theorem}[C.~Enache {\cite{En10}}]
Assume that $f' \ge 0$ and that
$u(\mathbf{x}) \in C^{3}(\Omega) \cap C^{2}(\overline{\Omega})$
is an admissible solution of \eqref{eq:1.1} with convex level sets.
Then the function
$\Phi \left(\mathbf{x},(N(N-1)/2)^{-1/2}\right)$
attains its maximum on $\partial \Omega$.
\end{theorem}

In dimension $N=2$, the $2$-Hessian operator reduces to the Monge--Amp\`{e}re
operator, and both maximum and minimum principles for $\Phi$ are available:

\begin{theorem}[C.~Enache {\cite{En14}}]
Assume that $N=2$ and that
$u(\mathbf{x}) \in C^{3}(\Omega) \cap C^{2}(\overline{\Omega})$
is an admissible solution of \eqref{eq:1.1}.
\begin{itemize}
\item[(i)]
If $f' \ge 0$ and
$\alpha \in (-\infty,-1] \cup [0,1]$,
then $\Phi(\mathbf{x},\alpha)$ attains its maximum on $\partial \Omega$.
\item[(ii)]
If $f' \le 0$ and
$\alpha \in [-1,0) \cup [1,\infty)$,
then $\Phi(\mathbf{x},\alpha)$ attains its minimum on $\partial \Omega$.
\end{itemize}
\end{theorem}

The purpose of this paper is to extend the minimum principle in
Theorem~2(ii) to higher dimensions, thereby obtaining a complementary result
to the maximum principle in Theorem~1.
Our main result shows that, under a natural convexity assumption on a
strictly monotone function $U(u(\mathbf{x}))$, the $P$-function $\Phi$ satisfies a
minimum principle in arbitrary dimension.

Throughout the paper, commas denote partial differentiation and
summation over repeated indices from $1$ to $N$ is understood.
We also use the notations
$u_{\min} := \min_{\overline{\Omega}} u$
and
$|\nabla u|_{\min}
:= \min_{\partial \Omega} |\nabla u|$.

More precisely, we prove the following:

\begin{theorem}\label{thm:main}
Assume that $f' \le 0$ and  $\alpha \in [1,\infty)$.
Let $u(\mathbf{x}) \in C^{3}(\Omega)\cap C^{2}(\overline{\Omega})$ be an
admissible solution of \eqref{eq:1.1}.
Assume further that there exists a strictly increasing $C^{2}$ function $U$
such that $U(u(\mathbf{x}))$ is convex in $\Omega$.
Then $\Phi(\mathbf{x},\alpha)$ attains its minimum value on $\partial\Omega$.
\end{theorem}

The proof relies on a sharp matrix inequality, established in Section~2,
and on Hopf's maximum principle~\cite{Ho27} applied to a suitable second
order differential inequality for $\Phi$, derived in Section~3.
In Section~4 we present several applications, yielding a priori bounds
for different classes of $2$-Hessian boundary value problems.

\section{A sharp matrix inequality and hessian estimates}\label{sec:matrix}

The following inequality plays a basic role in this investigation.

\begin{lemma}\label{lem:matrix}
Let $A=(a_{ik})$ be a real symmetric $N\times N$ matrix which is
\emph{positive semidefinite}. Let $\mathbf{v}=(v_{1},\dots,v_{N})^{t}\in\mathbb{R}^{N}$.
Then
\begin{equation}
\frac{1}{3}\|\mathbf{v}\|^{2}\Big\{2\,\mathrm{tr}(BA)-\mathrm{tr}(B)\,\mathrm{tr}(A)\Big\}
\le 2(A\mathbf{v},B\mathbf{v})-(A\mathbf{v},\mathbf{v})\,\mathrm{tr}(B),
\tag{2.1}\label{eq:2.1}
\end{equation}
where
\begin{equation}
B:=\mathrm{tr}(A)\,A-A^{2}.
\tag{2.2}\label{eq:2.2}
\end{equation}
Here $(\mathbf{v},\mathbf{w})=\sum_{k=1}^{N}v_{k}w_{k}$ denotes the Euclidean scalar product,
$\|\mathbf{v}\|=\sqrt{(\mathbf{v},\mathbf{v})}$ the corresponding norm, and $\mathrm{tr}(A)=\sum_{i=1}^{N}a_{ii}$
the trace.
\end{lemma}

\begin{proof}
Since $A$ is real symmetric and positive semidefinite, there exists an orthogonal matrix $P$ such that
\begin{equation}
A=P^{t}DP,
\tag{2.3}\label{eq:2.3}
\end{equation}
where
\begin{equation}
D=\mathrm{diag}(\lambda_{1},\dots,\lambda_{N}),
\qquad \lambda_{k}\ge 0 \ \text{for all }k,
\tag{2.4}\label{eq:2.4}
\end{equation}
and we define the vector
\begin{equation}
\mathbf{w}:=P\mathbf{v}.
\tag{2.5}\label{eq:2.5}
\end{equation}

We then compute the traces appearing in \eqref{eq:2.1}. First
\begin{equation}
\mathrm{tr}(A)=\mathrm{tr}(D)=\sum_{k=1}^{N}\lambda_{k}.
\tag{2.6}\label{eq:2.6}
\end{equation}
Then, from \eqref{eq:2.2}, we have
\begin{equation}
\mathrm{tr}(B)=\mathrm{tr}(A)^{2}-\mathrm{tr}(A^{2})
=2S_{2}(A),
\tag{2.7}\label{eq:2.7}
\end{equation}
and similarly
\begin{align}
2\,\mathrm{tr}(BA)
&=2\,\mathrm{tr}\big(\mathrm{tr}(A)A^{2}-A^{3}\big)
=2\,\mathrm{tr}(A)\,\mathrm{tr}(A^{2})-2\,\mathrm{tr}(A^{3}) \notag\\
&=\mathrm{tr}(B)\,\mathrm{tr}(A)-6S_{3}(A).
\tag{2.8}\label{eq:2.8}
\end{align}

Next, we compute
\begin{equation}
\|\mathbf{v}\|^{2}=\|\mathbf{w}\|^{2}=\sum_{k=1}^{N}w_{k}^{2},
\tag{2.9}\label{eq:2.9}
\end{equation}
\begin{equation}
(A\mathbf{v},\mathbf{v})=(D\mathbf{w},\mathbf{w})
=\sum_{k=1}^{N}\lambda_{k}w_{k}^{2},
\tag{2.10}\label{eq:2.10}
\end{equation}
and, since $B=\mathrm{tr}(A)A-A^{2}$ commutes with $A$ and is diagonalized by the same $P$,
\begin{align}
(A\mathbf{v},B\mathbf{v})
&=(D\mathbf{w},(\mathrm{tr}(D)D-D^{2})\mathbf{w}) \notag\\
&=\sum_{k=1}^{N}\lambda_{k}\big(\mathrm{tr}(D)\lambda_{k}-\lambda_{k}^{2}\big)w_{k}^{2}
=\sum_{k=1}^{N}\Big(\sum_{i\ne k}^{N}\lambda_{i}\Big)\lambda_{k}^{2}w_{k}^{2}.
\tag{2.11}\label{eq:2.11}
\end{align}

We now evaluate the difference ``right-hand side minus left-hand side'' in \eqref{eq:2.1}.
Using \eqref{eq:2.7}--\eqref{eq:2.11} we obtain
\begin{equation}
\begin{alignedat}{2}
& 2(A\mathbf{v},B\mathbf{v})-(A\mathbf{v},\mathbf{v})\,\mathrm{tr}(B)
-\frac{1}{3}\|\mathbf{v}\|^{2}
\Big\{2\,\mathrm{tr}(BA)-\mathrm{tr}(B)\,\mathrm{tr}(A)\Big\} \\[10pt]
&\qquad
= 2\sum_{k=1}^{N}\Big(\sum_{i\ne k}^{N}\lambda_{i}\Big)\lambda_{k}^{2}w_{k}^{2}
-2S_{2}(A)\sum_{k=1}^{N}\lambda_{k}w_{k}^{2}
+2S_{3}(A)\sum_{k=1}^{N}w_{k}^{2} \\[10pt]
&\qquad
= 2\sum_{k=1}^{N}w_{k}^{2}
\Big(S_{3}(A)-\lambda_{k}S_{2}^{(k)}(A)\Big)
= 2\sum_{k=1}^{N}w_{k}^{2}S_{3}^{(k)}(A).
\end{alignedat}
\tag{2.12}\label{eq:2.12}
\end{equation}
where $S_{2}^{(k)}(A)$ denotes the second elementary symmetric function of the $(N-1)$-tuple
$(\lambda_{1},\dots,\widehat{\lambda_{k}},\dots,\lambda_{N})$ (i.e. with $\lambda_k$ omitted), and we set
\begin{equation}
S_{3}^{(k)}(\lambda_{1},\dots,\lambda_{N})
:=
\begin{cases}
S_{3}(\lambda_{1},\dots,\widehat{\lambda_{k}},\dots,\lambda_{N}), & N\ge 4,\\[1mm]
0, & N=3.
\end{cases}
\tag{2.13}\label{eq:2.13}
\end{equation}
If $N\ge 4$, then $S_{3}^{(k)}(\lambda)\ge 0$ since $\lambda_j\ge 0$ for all $j$.
Hence the right-hand side of \eqref{eq:2.12} is nonnegative, which proves \eqref{eq:2.1}.
If $N=3$, \eqref{eq:2.12} shows that \eqref{eq:2.1} reduces to an identity for all symmetric $3\times 3$ matrices.
\end{proof}

\begin{remark}\label{rem:neg_semidef}
If $A$ is \emph{negative semidefinite}, then inequality \eqref{eq:2.1} holds with the opposite sign.
Indeed, all eigenvalues satisfy $\lambda_{k}\le 0$, hence for $N\ge 4$ we have
$S_{3}^{(k)}(\lambda)\le 0$, and the right-hand side of \eqref{eq:2.12} is nonpositive.
\end{remark}

We now state an important consequence of Lemma~\ref{lem:matrix} needed later.

\begin{lemma}\label{lem:Uu}
Let $u\in C^{2}(\Omega)$ and assume that there exists a strictly increasing function
$U\in C^{2}(I)$, defined on an interval $I\supset u(\Omega)$, such that the composition
\[
\widetilde U(\mathbf{x}) := U(u(\mathbf{x}))
\]
is convex in $\Omega\subset\mathbb{R}^{N}$, $N\ge2$.
Then, for every $\mathbf{x}\in\Omega$, the following inequality holds:
\begin{align}
\frac{1}{3}|\nabla u|^{2}
\Big[
2S_{2}^{ij}(D^{2}u)\,u_{,jl}u_{,li}
-2S_{2}(D^{2}u)\,\Delta u
\Big]
\;\le\;& \notag\\
2S_{2}^{ij}(D^{2}u)\,u_{,jk}u_{,k}\,u_{,i\ell}u_{,\ell}
-2S_{2}(D^{2}u)\,u_{,ik}u_{,i}u_{,k}.&
\tag{2.14}\label{eq:2.14}
\end{align}
\end{lemma}

\begin{proof}
Since $\widetilde U$ is convex in $\Omega$, its Hessian matrix
$A:=D^{2}\widetilde U(\mathbf{x})$ is positive semidefinite for every $\mathbf{x}\in\Omega$.
We apply Lemma~\ref{lem:matrix} with this choice of $A$ and with $\mathbf{v}:=\nabla u(\mathbf{x})$.
This yields
\begin{equation}
\begin{aligned}
M(\mathbf{x}) :=\ &
\frac{1}{3}|\nabla u|^{2}
\Big[
2S_{2}^{ij}(D^{2}\widetilde U)\,\widetilde U_{,jl}\widetilde U_{,li}
-2S_{2}(D^{2}\widetilde U)\,\Delta \widetilde U
\Big]\\
&-2S_{2}^{ij}(D^{2}\widetilde U)\,\widetilde U_{,jk}u_{,k}\,
\widetilde U_{,i\ell}u_{,\ell}
+2S_{2}(D^{2}\widetilde U)\,\widetilde U_{,ik}u_{,i}u_{,k}
\;\le\;0,
\end{aligned}
\tag{2.15}\label{eq:2.15}
\end{equation}
for all $\mathbf{x}\in\Omega$.

\medskip\noindent
\textbf{Claim.} For every $\mathbf{x}\in\Omega$ we have the identity
\begin{equation}
\begin{aligned}
M(\mathbf{x})=\ &U'(u(\mathbf{x}))^{3}\Big\{
\frac{1}{3}|\nabla u|^{2}\Big[2S_{2}^{ij}(D^{2}u)\,u_{,jl}u_{,li}
-2S_{2}(D^{2}u)\,\Delta u\Big]\\
&\hspace{20mm}
-2S_{2}^{ij}(D^{2}u)\,u_{,jk}u_{,k}\,u_{,i\ell}u_{,\ell}
+2S_{2}(D^{2}u)\,u_{,ik}u_{,i}u_{,k}\Big\}.
\end{aligned}
\tag{2.16}\label{eq:2.16}
\end{equation}
Assuming the claim, inequality \eqref{eq:2.14} follows immediately from \eqref{eq:2.15}--\eqref{eq:2.16},
since $U'(u)>0$ by strict monotonicity.

\medskip\noindent
\textbf{Proof of the claim.}
Fix $\mathbf{x}_{0}\in\Omega$. Choose an orthonormal frame at $\mathbf{x}_{0}$ such that
\begin{equation}
A:=D^{2}u(\mathbf{x}_{0})=\mathrm{diag}(\lambda_{1},\dots,\lambda_{N}).
\tag{2.17}\label{eq:2.17}
\end{equation}
Set
\begin{equation}
\alpha:=U'(u(\mathbf{x}_{0})),\qquad
\beta:=U''(u(\mathbf{x}_{0})),\qquad
v:=\nabla u(\mathbf{x}_{0}),
\tag{2.18}\label{eq:2.18}
\end{equation}
and introduce the scalars
\begin{equation}
r:=|v|^{2},\qquad
s:=\Delta u(\mathbf{x}_{0})=\mathrm{tr}(A),\qquad
q:=\langle Av,v\rangle,\qquad
t:=|Av|^{2}.
\tag{2.19}\label{eq:2.19}
\end{equation}

Since $\widetilde U(\mathbf{x})=U(u(\mathbf{x}))$, we have at $\mathbf{x}_{0}$:
\begin{equation}
D^{2}\widetilde U=\alpha A+\beta (v\otimes v),\qquad
\Delta\widetilde U=\alpha s+\beta r.
\tag{2.20}\label{eq:2.20}
\end{equation}
Let $B:=v\otimes v$. Then
\begin{equation}
B^{2}=rB,\qquad \mathrm{tr}(B)=r,\qquad \mathrm{tr}(AB)=q.
\tag{2.21}\label{eq:2.21}
\end{equation}

We use the standard identities valid for every real $N\times N$ matrix $M$:
\begin{equation}
S_{2}(M)=\tfrac12\big[(\mathrm{tr}\,M)^{2}-\mathrm{tr}(M^{2})\big],
\qquad
S_{2}^{ij}(M)=(\mathrm{tr}\,M)\delta_{ij}-M_{ij}.
\tag{2.22}\label{eq:2.22}
\end{equation}
A direct computation based on \eqref{eq:2.20}--\eqref{eq:2.22} yields
\begin{equation}
S_{2}(D^{2}\widetilde U)=\alpha^{2}S_{2}(A)+\alpha\beta(sr-q),
\tag{2.23}\label{eq:2.23}
\end{equation}
and
\begin{equation}
S_{2}^{ij}(D^{2}\widetilde U)=\alpha S_{2}^{ij}(A)+\beta\,(r\delta_{ij}-v_{i}v_{j}).
\tag{2.24}\label{eq:2.24}
\end{equation}

At $\mathbf{x}_{0}$, expression \eqref{eq:2.15} can be rewritten as
\begin{equation}
\begin{aligned}
M(\mathbf{x}_{0})=\ &
\frac{1}{3}r\Big(2S_{2}^{ij}(D^{2}\widetilde U)\,\widetilde U_{,jl}\widetilde U_{,li}
-2S_{2}(D^{2}\widetilde U)\,\Delta\widetilde U\Big)\\
&-2S_{2}^{ij}(D^{2}\widetilde U)\,(\widetilde U v)_{j}(\widetilde U v)_{i}
+2S_{2}(D^{2}\widetilde U)\,\langle \widetilde U v,v\rangle,
\end{aligned}
\tag{2.25}\label{eq:2.25}
\end{equation}
where $\widetilde U:=D^{2}\widetilde U(\mathbf{x}_{0})$ and $(\widetilde U v)_{j}=\widetilde U_{,jk}v_{k}$.

Now expand $\widetilde U=\alpha A+\beta B$ in \eqref{eq:2.25} and we group terms according to the coefficients of the powers of $\alpha$ and $\beta$:
\begin{equation}
M(\mathbf{x}_{0})=\alpha^{3}M_{3,0}+\alpha^{2}\beta M_{2,1}+\alpha\beta^{2}M_{1,2}+\beta^{3}M_{0,3}.
\tag{2.26}\label{eq:2.26}
\end{equation}
In what follows we analyze separately each contribution in \eqref{eq:2.26}.

\medskip\noindent
\emph{Step 1: the coefficient of $\alpha^{3}$.}
Keeping only the $\alpha^{3}$ part (i.e. setting $\beta=0$) we obtain
\begin{equation}
\begin{aligned}
M_{3,0}=\ &
\tfrac13 r\Big(2S_{2}^{ij}(A)(A^{2})_{ji}-2S_{2}(A)\,s\Big)
-2S_{2}^{ij}(A)(Av)_{j}(Av)_{i}
+2S_{2}(A)\,q,
\end{aligned}
\tag{2.27}\label{eq:2.27}
\end{equation}
which is exactly the bracket in \eqref{eq:2.16} evaluated at $\mathbf{x}_{0}$.

\medskip\noindent
\emph{Step 2: the coefficient of $\alpha^{2}\beta$ vanishes.}
Using \eqref{eq:2.23}--\eqref{eq:2.24} and the elementary contractions
\begin{equation}
S_{2}^{ij}(A)(AB+BA)_{ji}=2sq-2t,\qquad
(r\delta_{ij}-v_{i}v_{j})(A^{2})_{ji}=r\,\mathrm{tr}(A^{2})-t,
\tag{2.28}\label{eq:2.28}
\end{equation}
together with
\begin{equation}
S_{2}^{ij}(A)(Av)_{j}v_{i}=sq-t,\qquad
(r\delta_{ij}-v_{i}v_{j})(Av)_{j}(Av)_{i}=rt-q^{2},
\tag{2.29}\label{eq:2.29}
\end{equation}
one checks by direct substitution into \eqref{eq:2.25} that all $\alpha^{2}\beta$ terms cancel, hence
\begin{equation}
M_{2,1}=0.
\tag{2.30}\label{eq:2.30}
\end{equation}

\medskip\noindent
\emph{Step 3: the coefficients of  $\alpha\beta^{2}$ and $\beta^{3}$ vanish.}
These terms involve contractions of $(r\delta_{ij}-v_{i}v_{j})$ with expressions built from $B$ and $AB+BA$.
Using $B^{2}=rB$ and the identities
\begin{equation}
(r\delta_{ij}-v_{i}v_{j})(AB+BA)_{ji}=0,\qquad
(r\delta_{ij}-v_{i}v_{j})(B^{2})_{ji}=0,
\tag{2.31}\label{eq:2.31}
\end{equation}
we get
\begin{equation}
M_{1,2}=M_{0,3}=0.
\tag{2.32}\label{eq:2.32}
\end{equation}

\medskip\noindent
Finally, combining now \eqref{eq:2.26}--\eqref{eq:2.32} we obtain
\[
M(\mathbf{x}_{0})=\alpha^{3}M_{3,0}
=U'(u(\mathbf{x}_{0}))^{3}\times\Big(\text{the bracket in \eqref{eq:2.16} at }\mathbf{x}_{0}\Big).
\]
Since $\mathbf{x}_{0}$ is arbitrary, identity \eqref{eq:2.16} holds for all $\mathbf{x}\in\Omega$.
This completes the proof of the claim and hence of the lemma.
\end{proof}

\begin{remark}\label{rem:decreasingU}
If $U$ is strictly \emph{decreasing} (so $U'(u)<0$), the same computation still yields
the identity \eqref{eq:2.16}, but now $U'(u)^{3}<0$. Consequently, inequality \eqref{eq:2.14}
holds with the \emph{opposite} sign.
\end{remark}

\section{Proof of the main result}

For the proof of Theorem~3 we derive a suitable elliptic differential inequality
for the $P$--function $\Phi(\mathbf{x};\alpha)$ defined in \eqref{eq:1.5},
and then apply Hopf's first maximum principle \cite{Ho27}.
Differentiating \eqref{eq:1.5} we obtain
\begin{equation}
\Phi_{,k}=2u_{,kj}u_{,j}-2\alpha f^{1/2}(u)\,u_{,k},
\tag{3.1}\label{eq:3.1}
\end{equation}
and
\begin{equation}
\begin{aligned}
S_2^{kl}\Phi_{,kl}
&=2S_2^{kl}u_{,klj}u_{,j}+2S_2^{kl}u_{,kj}u_{,jl}
-\alpha f^{-1/2}(u)f'(u)\,S_2^{kl}u_{,k}u_{,l}\\
&\quad -2\alpha f^{1/2}(u)\,S_2^{kl}u_{,kl},
\end{aligned}
\tag{3.2}\label{eq:3.2}
\end{equation}
where $(S_2^{kl})$ is the positive definite matrix defined in \eqref{eq:1.4}.

We compute separately the terms in \eqref{eq:3.2}.
Since $S_2$ is homogeneous of degree $2$, Euler's identity together with \eqref{eq:1.1} yields
\begin{equation}
S_2^{kl}u_{,kl}=2S_2(D^2u)=2f(u).
\tag{3.3}\label{eq:3.3}
\end{equation}

Next, on each level set $L_t=\{u=t\}$ we recall the identity (see Brandolini et al.~\cite{BNST08})
\begin{equation}
S_2^{ij}u_{,i}u_{,l}u_{,lj}
=S_2(D^2u)|\nabla u|^2-H_2|\nabla u|^3,
\tag{3.4}\label{eq:3.4}
\end{equation}
where $H_2$ denotes the second mean curvature of $L_t$, i.e. the second elementary symmetric function of the principal curvatures of $L_t$.

Using \eqref{eq:3.1} we have $u_{,lj}u_{,l}=\alpha f^{1/2}(u)\,u_{,j}+\tfrac12\Phi_{,j}$, hence \eqref{eq:3.4} gives
\begin{equation}
S_2^{kl}u_{,k}u_{,l}
=\frac{1}{\alpha f^{1/2}(u)}
\Big[S_2(D^2u)|\nabla u|^2-H_2|\nabla u|^3\Big]+\cdots,
\tag{3.5}\label{eq:3.5}
\end{equation}
in $\Omega\setminus\{\mathbf{K}\}$, where $\mathbf{K}$ denotes the unique critical point of $u$ and the dots stand, here and throughout the remainder of the paper, for terms involving first-order derivatives of $\Phi$. Note that the uniqueness of $\mathbf{K}$ is a consequence of the assumption that $U(u)$ is convex.

Differentiating \eqref{eq:1.1} and using the chain rule we obtain
\begin{equation}
S_2^{kl}u_{,klj}u_{,j}=f'(u)|\nabla u|^2.
\tag{3.6}\label{eq:3.6}
\end{equation}

To estimate $S_2^{kl}u_{,kj}u_{,jl}$ we use inequality \eqref{eq:2.14} from Section~2:
\begin{equation}
\begin{aligned}
\frac{1}{3}|\nabla u|^2
\Big[2S_2^{ij}u_{,jl}u_{,li}-2S_2(D^2u)\Delta u\Big]
\le\;&
2S_2^{ij}u_{,kj}u_{,k}\,u_{,i\ell}u_{,\ell}\\
&-2S_2(D^2u)\,u_{,ik}u_{,i}u_{,k}.
\end{aligned}
\tag{3.7}\label{eq:3.7}
\end{equation}

Using \eqref{eq:3.1} we infer
\begin{equation}
u_{,ij}u_{,i}u_{,j}=\alpha f^{1/2}(u)|\nabla u|^2+\cdots,
\tag{3.8}\label{eq:3.8}
\end{equation}
and, combining \eqref{eq:3.1} with \eqref{eq:3.4},
\begin{equation}
S_2^{ij}u_{,kj}u_{,k}\,u_{,i\ell}u_{,\ell}
=\alpha f^{1/2}(u)\Big[S_2(D^2u)|\nabla u|^2-H_2|\nabla u|^3\Big]+\cdots.
\tag{3.9}\label{eq:3.9}
\end{equation}

Substituting \eqref{eq:3.8}--\eqref{eq:3.9} into \eqref{eq:3.7} we obtain
\begin{equation}
2S_2^{kl}u_{,kj}u_{,jl}
\le 2f(u)\Delta u
-6\alpha f^{1/2}(u)\,H_2|\nabla u|
+\cdots,
\tag{3.10}\label{eq:3.10}
\end{equation}
in $\Omega\setminus\{\mathbf{K}\}$.

Collecting \eqref{eq:3.3}, \eqref{eq:3.5}, \eqref{eq:3.6} and \eqref{eq:3.10} in \eqref{eq:3.2} we deduce
\begin{equation}
\begin{aligned}
S_2^{kl}\Phi_{,kl}\ge\;&
f'(u)|\nabla u|^2
+2f(u)\Delta u
-6\alpha f^{1/2}(u)\,H_2|\nabla u|\\
&\quad
+\frac{f'(u)}{f(u)}H_2|\nabla u|^3
-4\alpha f^{3/2}(u)
+\cdots,
\end{aligned}
\tag{3.11}\label{eq:3.11}
\end{equation}
in $\Omega\setminus\{\mathbf{K}\}$.

We next use the inequality of Philippin--Safoui~\cite{PS04}
\begin{equation}
|\nabla u|^2S_2(D^2u)
\ge
u_{,il}u_{,i}u_{,l}\Delta u-u_{,ik}u_{,k}u_{,il}u_{,l},
\tag{3.12}\label{eq:3.12}
\end{equation}
together with \eqref{eq:3.8}, to obtain
\begin{equation}
\Delta u\le \frac{1}{\alpha}f^{1/2}(u)+\alpha f^{1/2}(u)+\cdots.
\tag{3.13}\label{eq:3.13}
\end{equation}

Substituting \eqref{eq:3.13} into \eqref{eq:3.11}, we arrive at
\begin{equation}
\begin{aligned}
L\Phi
:=\; S_2^{kl}\Phi_{,kl}+W_k\Phi_{,k}
\;\le \;&
f'(u)|\nabla u|^2
+\Big(\frac{2}{\alpha}-2\alpha\Big)f^{3/2}(u) \\[4pt]
&\;
- H_2\Big(
6\alpha f^{1/2}(u)\,|\nabla u|
-\frac{f'(u)}{f(u)}|\nabla u|^3
\Big),
\end{aligned}
\tag{3.14}\label{eq:3.14}
\end{equation}
in $\Omega\setminus\{\mathbf{K}\}$, where $W_k$ is a smooth vector field.

It then follows from \eqref{eq:3.14} and Hopf's first maximum principle \cite{Ho27}
that the $P$--function $\Phi(\mathbf{x};\alpha)$ attains its \emph{minimum}
either on $\partial\Omega$ or at $\mathbf{K}\in\Omega$, unless it is constant
on $\overline{\Omega}$.

We now analyze separately the following subcases.

\medskip
\noindent\textbf{I.1. The subcase $\alpha>1$.}\medskip

Assume by contradiction that the minimum of $\Phi(\mathbf{x};\alpha)$
is attained at the (unique) critical point of $u$.
By a translation and a rotation if necessary, we may assume that this point
coincides with the origin $\mathbf{O}$ and that
$u_{,ij}(\mathbf{O})=0$ for all $i\neq j$.
Then
\begin{equation}
\Phi_{,ii}(\mathbf{O},\alpha)
=2u_{,ii}(\mathbf{O})\big[u_{,ii}(\mathbf{O})-\alpha f^{1/2}(\mathbf{O})\big],
\qquad i=1,\dots,N.
\tag{3.16}\label{eq:3.16}
\end{equation}

Since $u$ has convex level sets and $\mathbf{O}$ is the unique critical point of u, it follows that $\mathbf{O}$ is a strict interior minimum of u. Consequently, the Hessian matrix $D^2u(\mathbf{O})$ is positive definite, hence
\begin{equation}
u_{,ii}(\mathbf{O})>0,
\qquad i=1,\dots,N.
\tag{3.17}\label{eq:3.17}
\end{equation}
Moreover, because $\mathbf{O}$ is a minimum point of $\Phi(\cdot;\alpha)$,
\begin{equation}
\Phi_{,ii}(\mathbf{O},\alpha)\le 0,
\qquad i=1,\dots,N.
\tag{3.18}\label{eq:3.18}
\end{equation}
Combining \eqref{eq:3.16}--\eqref{eq:3.18}, we obtain
\begin{equation}
u_{,ii}(\mathbf{O})\le \alpha f^{1/2}(\mathbf{O}),
\qquad i=1,\dots,N.
\tag{3.19}\label{eq:3.19}
\end{equation}
Hence
\begin{equation}
S_2(D^2u)(\mathbf{O})
\le \binom{N}{2}\alpha^2 f(\mathbf{O}).
\tag{3.20}\label{eq:3.20}
\end{equation}
Using \eqref{eq:1.1}, this implies
\begin{equation}
\alpha\le \binom{N}{2}^{-1/2}\le 1,
\tag{3.21}\label{eq:3.21}
\end{equation}
which contradicts the assumption $\alpha>1$.

Therefore, for $\alpha>1$, the auxiliary function $\Phi(\mathbf{x};\alpha)$
attains its \emph{minimum} on the boundary $\partial\Omega$, unless it is
constant on $\overline{\Omega}$.
Since the strict inequality in \eqref{eq:3.14} (with $f'\le0$ and $H_2\ge0$)
prevents $\Phi(\cdot;\alpha)$ from being constant, we conclude that
\begin{equation}
\min_{\overline{\Omega}}\Phi(\mathbf{x};\alpha)
=\min_{\partial\Omega}\Phi(\mathbf{x};\alpha).
\tag{3.22}\label{eq:3.22}
\end{equation}

\medskip
\noindent\textbf{I.2. The subcase $\alpha=1$.}\medskip

This limiting case is handled by a continuity argument.
From the previous analysis we know that, for each $\alpha>1$, 
$\Phi(\mathbf{x};\alpha)$ attains its minimum on $\partial\Omega$.
As $\alpha$ decreases continuously from values larger than $1$ down to $1$, 
the location of the minimum point moves continuously along $\partial\Omega$
(or possibly remains fixed).

By Hopf's maximum principle applied to \eqref{eq:3.14}, 
$\Phi(\mathbf{x};1)$ can attain its minimum only at the unique
critical point $\mathbf{K}\in\Omega$ or on $\partial\Omega$, unless it is constant
on $\overline{\Omega}$.
Since the minimum cannot jump discontinuously from $\partial\Omega$ (for $\alpha>1$)
to the interior point $\mathbf{K}$ (for $\alpha=1$), we conclude by continuity that
$\Phi(\mathbf{x};1)$ attains its minimum on $\partial\Omega$.

\section{Applications: a priori estimates}\label{sec:applications}

In this section we apply Theorem~\ref{thm:main} to derive a priori bounds 
for several classical 2-Hessian boundary value problems. The key ingredient 
is the verification that the convexity hypothesis in Lemma~\ref{lem:Uu} 
holds for suitable strictly increasing transformations $U$. 

Such convexity and power-convexity properties have been established by 
X.-N. Ma and L. Xu \cite{MX08}, P. Liu, X.-N. Ma and L. Xu \cite{LMX10}, 
P. Salani \cite{Sa12}, and Y. Ye \cite{Ye13} for various classes of 
2-Hessian equations using different techniques. These results enable us 
to apply our minimum principle and obtain explicit bounds.

\medskip
\noindent
\textbf{Application 1.}
Let $u$ be the solution of
\begin{equation}
S_{2}(D^{2}u)=1 \quad \text{in }\Omega,
\qquad
u=0 \quad \text{on }\partial\Omega,
\tag{4.1}\label{eq:4.1}
\end{equation}
where $\Omega$ is a bounded convex domain in $\mathbb{R}^{3}$.
It is well known that $u<0$ in $\Omega$.
X.-N.~Ma and L.~Xu proved in \cite{MX08} that the function
\begin{equation}
\widetilde U(\mathbf{x}) := U(u(\mathbf{x})),
\qquad
U(t):=-\sqrt{-t},
\tag{4.2}\label{eq:4.2}
\end{equation}
is strictly convex in $\Omega$.
Note that $U\in C^{2}((-\infty,0))$ and $U'(t)=\frac{1}{2}(-t)^{-1/2}>0$ for $t<0$,
so $U$ is strictly increasing on $u(\Omega)\subset(-\infty,0)$.
Hence the assumptions of Lemma~\ref{lem:Uu} are satisfied.

It follows from Theorem~3 that, for every $\alpha\ge1$,
\begin{equation}
\Phi(\mathbf{x};\alpha)
:=|\nabla u(\mathbf{x})|^{2}-2\alpha u(\mathbf{x})
\tag{4.3}\label{eq:4.3}
\end{equation}
attains its minimum value on $\partial\Omega$.
Choosing $\alpha=1$ and using that
$\min_{\partial\Omega}\Phi(\mathbf{x};1)=|\nabla u|_{\min}^{2}$, we obtain
\begin{equation}
-2u(\mathbf{x})
\ge
|\nabla u|_{\min}^{2}-|\nabla u(\mathbf{x})|^{2},
\qquad \mathbf{x}\in\Omega.
\tag{4.4}\label{eq:4.4}
\end{equation}

Evaluating \eqref{eq:4.4} at the (unique) critical point of $u$, where
$\nabla u=0$, we deduce
\begin{equation}
-2u_{\min}\ge |\nabla u|_{\min}^{2}.
\tag{4.5}\label{eq:4.5}
\end{equation}
In dimension $N=2$, equality in \eqref{eq:4.4}--\eqref{eq:4.5} holds if and only if
$\Omega$ is a disk (see \cite{EMP21}).

\medskip
\noindent
\textbf{Application 2.}
Let $u$ be the first eigenfunction of the $2$-Hessian operator, that is,
\begin{equation}
S_{2}(D^{2}u)=\lambda_{1}(-u)^{2} \quad \text{in }\Omega,
\qquad
u=0 \quad \text{on }\partial\Omega,
\tag{4.6}\label{eq:4.6}
\end{equation}
where $\Omega$ is a bounded convex domain in $\mathbb{R}^{3}$.
Then $u<0$ in $\Omega$.
P.~Liu, X.-N.~Ma and L.~Xu showed in \cite{LMX10} that
\begin{equation}
\widetilde U(\mathbf{x}) := U(u(\mathbf{x})),
\qquad
U(t):=-\log(-t),
\tag{4.7}\label{eq:4.7}
\end{equation}
is strictly convex in $\Omega$.
Since $U\in C^{2}((-\infty,0))$ and $U'(t)=-\frac{1}{t}>0$ for $t<0$,
the hypotheses of Lemma~\ref{lem:Uu} are satisfied.

Applying Theorem~3 with $\alpha=1$, we deduce that
\begin{equation}
\Phi(\mathbf{x};1)
=|\nabla u(\mathbf{x})|^{2}-\frac{2}{3}\lambda_{1}u(\mathbf{x})^{3}
\tag{4.8}\label{eq:4.8}
\end{equation}
attains its minimum value on $\partial\Omega$.
Consequently,
\begin{equation}
-\frac{2}{3}\lambda_{1}u(\mathbf{x})^{3}
\ge
|\nabla u|_{\min}^{2}-|\nabla u(\mathbf{x})|^{2},
\qquad \mathbf{x}\in\Omega,
\tag{4.9}\label{eq:4.9}
\end{equation}
and in particular, at the critical point of $u$,
\begin{equation}
-\frac{2}{3}\lambda_{1}u_{\min}^{3}
\ge
|\nabla u|_{\min}^{2}.
\tag{4.10}\label{eq:4.10}
\end{equation}

\medskip
\noindent
\textbf{Application 3.}
Let $u$ be the solution of
\begin{equation}
S_{2}(D^{2}u)=\lambda(-u)^{p} \quad \text{in }\Omega,
\qquad
u=0 \quad \text{on }\partial\Omega,
\tag{4.11}\label{eq:4.11}
\end{equation}
where $\Omega$ is a bounded convex domain in $\mathbb{R}^{N}$,
$\lambda>0$, and $u<0$ in $\Omega$.

Assume first that $0<p<2$.
Y.~Ye proved in \cite{Ye13} that the function
\begin{equation}
\widetilde U(\mathbf{x}) := U(u(\mathbf{x})),
\qquad
U(t):=-(-t)^{\frac{2-p}{4}},
\tag{4.12}\label{eq:4.12}
\end{equation}
is strictly convex in $\Omega$.
Since $\frac{2-p}{p}>0$, we have $U'(t)>0$ for $t<0$, hence $U$ is strictly increasing
on $u(\Omega)\subset(-\infty,0)$, and Lemma~\ref{lem:Uu} applies.

Applying Theorem~3 with $\alpha=1$, we conclude that
\begin{equation}
\Phi(\mathbf{x};1)
=
|\nabla u(\mathbf{x})|^{2}
+\frac{2\lambda}{p+1}(-u(\mathbf{x}))^{p+1},
\tag{4.13}\label{eq:4.13}
\end{equation}
attains its minimum value on $\partial\Omega$.
As a consequence,
\begin{equation}
\frac{2\lambda}{p+1}(-u(\mathbf{x}))^{p+1}
\ge
|\nabla u|_{\min}^{2}-|\nabla u(\mathbf{x})|^{2},
\qquad \mathbf{x}\in\Omega,
\tag{4.14}\label{eq:4.14}
\end{equation}
and at the critical point of $u$,
\begin{equation}
\frac{2\lambda}{p+1}(-u_{\min})^{p+1}
\ge
|\nabla u|_{\min}^{2}.
\tag{4.15}\label{eq:4.15}
\end{equation}

\begin{remark}
When $p>2$, the exponent $\frac{2-p}{4}$ is negative and the function
$U(t)=-(-t)^{(2-p)/4}$ is strictly \emph{decreasing} on $(-\infty,0)$.
In this case Lemma~\ref{lem:Uu} (as stated with $U'>0$) cannot be applied directly.
\end{remark}

\end{document}